\documentclass[10pt]{article}
\usepackage{amssymb,amsmath,amsthm}

\usepackage{mathrsfs}

\usepackage{fullpage}
\usepackage{graphicx}
\usepackage{color}

\newtheorem{theorem}{Theorem}
\newtheorem{lemma}[theorem]{Lemma}
\newtheorem{corollary}[theorem]{Corollary}

\newtheorem{proposition}[theorem]{Proposition}
\newtheorem{remark}[theorem]{Remark}

\title{Mixed Cages: monotony, connectivity and upper bounds 
\thanks{Research supported by CONACyT-M{\' e}xico under Project 282280 and PAPIIT-M{\' e}xico under Projects IN107218, IN106318.}}
\author{Gabriela Araujo-Pardo\footnotemark[2] \and Claudia de la Cruz\footnotemark[3] \and Diego Gonz{\' a}lez-Moreno \footnotemark[4]}

\begin{document}
\maketitle

\def\thefootnote{\fnsymbol{footnote}}
\footnotetext[2]{Instituto de Matem{\' a}ticas, Universidad Nacional Aut{\'o}noma de M{\' e}xico, Campus Juriquilla, Quer{\' e}taro.{\tt garaujo@math.unam.mx}.}
\footnotetext[3]{Universidad Aut{\' o}noma Metropolitana.{\tt clau.mar@ciencias.unam.mx}.}
\footnotetext[4]{Departamento de Matem{\' a}ticas Aplicadas y Sistemas, UAM-Cuajimalpa, Mexico City, Mexico. {\tt
dgonzalez@correo.cua.uam.mx}.}

\begin{abstract} 
A \emph{$[z, r; g]$-mixed cage} is a mixed graph $z$-regular by arcs, $r$-regular by edges, with girth $g$ and minimum order. 
Let $n[z,r;g]$ denote the order of a $[z,r;g]$-mixed cage. 

In this paper we prove that $n[z,r;g]$ is a monotonicity function, with respect of $g$, for $z\in \{1,2\}$, and we use it to prove that the underlying graph of a $[z,r;g]$-mixed cage is 2-connected, for $z\in \{1,2\}$. 
We also prove that  $[z,r;g]$-mixed cages are strong connected.
We present bounds of $n[z,r;g]$ and 
constructions of $[z,r;5]$-mixed graphs and show a $[10,3;5]$-mixed cage of order $50$.

\end{abstract}

\section{Introduction}

In this paper we consider graphs which are finite and mixed, 
that is, they may contain (directed) arcs as well as (undirected) edges. We don't allow multiple edges and arcs.

The mixed regular graphs were introduced in \cite{AHM19}. A {\it{mixed regular graph}} is a simple and finite graph $G$, such that for every $v\in V(G)$,  $v$ is the head of $z$ arcs, the tail of $z$ arcs and is incident with $r$ edges. The directed degree of a vertex $v$ is equal to $z$, while the undirected degree is equal to $r$. We set $d=z+r$ to be the \emph{degree} of $v$. 
We will consider walks of the form $(v_0, \dots, v_n)$, where eihter $v_i v_{i+1}$ is an edge of $G$ or $(v_i, v_{i+1})$ is an arc of $G$, for $i \in \{0, \dots, n-1\}$.   In other words, the walks could contain edges and arcs, provided that all the arcs are traversed in the same direction.
The \emph{girth} of $G$ is the length of the shortest cycle of $G$, we denote the length of a cycle $C$ as $\ell(C)$.
The \emph{distance} between two vertices $u$ and $v$ denoted by $d(u,v)$ is defined as the shortest length of all $uv$-paths.
If $G$ has girth equal to $g$, then  $G$ is a $[z,r;g]$-{\em{mixed graph}} of directed degree $z$, undirected degree $r$ and girth $g$. A $[z, r; g]$-{\em mixed cage} is a $[z, r; g]$-mixed graph of minimum order. 
Through this paper we use $n[z,k;g]$  to denote the order of a 
$[z, r; g]$-mixed graph.

The {\it{Cage Problem}} is to find  the smallest  number, $n(k,g)$, of vertices for a $k$-regular graph of girth $g$. It has been widely studied since cages were introduced by Tutte \cite{T47} in 1947 and after Erd\"os and Sachs  \cite{ES63} proved, in 1963, their existence. A complete survey about this topic and its relevance can be found in \cite{EJ08}. Moreover, there exists a lot of results, related with this problem that studied structural properties of cages as monotonicity and connectivity. For instance, Fu, Huang and Rodger \cite{Fu} proved that if $k\ge 2$ and  $3\leq g_1 < g_2$, then $n(k,g_1)<n(k,g_2)$. Concerning the connectivity, it is known that if $G$ is a $(k,g)$-cage, then $\lambda(G)=k$ \cite{LMC05,WXW03}, and for  every odd girth $g\ge 7$, $\kappa(G)\ge \lfloor k/2\rfloor +1$  \cite{conexidad}.  For even girth $g\ge 6$, $(k,g)$-cages with $k\ge 3$ are $(t+1)$-connected, $t$ being the largest integer such that $t^2+2t^2\le k$ \cite{LMB05}.

In this paper we are interested in the {\em Mixed Cage Problem}, that is, find constructions of $[z,r;g]$-mixed regular graphs, with specified degrees $z$, $r$, girth $g$ and minimum order.  
The work concerning constructions of $[z,r;5]$-mixed cages starts in \cite{AHM19} and continue in \cite{AB19} where the authors give constructions of small $[z,r;5]$-mixed cages with similar techniques than used in \cite{AABB17} to construct small regular graphs of girth five. 

The paper is organized as follows: in Section 2 we study structural properties of mixed cages. We prove that  $n[z,r;g]$ is a monotonicity function with respect of $g$, for $z\in \{1,2\}$. As a consequence of this result we show  that the underlying graph of a $[z,r;g]$-mixed cages is 2-connected, for $z\in\{1,2\}$. We also show that every $[z,r;g]$-mixed cage is strong connected.

In Section 3, we present a  lower bound for $n[z,r;g]$, specifically we show that if $g\geq 5$, then $n[z,r;g]\geq n_0(r,g)+2z$, where  $n_0(r,g)$ is the More bound.  We give two different constructions that provide us new upper bounds for $n[z,k;g]$. In the first construction we use the incident finite graph of a partial plane defined over any finite field generated by a prime. 
This construction was also used in previous papers, for example, to construct regular graphs of girth 5 \cite{Articulo_5} and to construct Mixed Moore graphs of diameter 2 \cite{ABMM16}. 
With this construction we state that $n[z,r;5]\leq 10zr$. 
In the second construction, we establish  that $n[z,r;g]\leq gn_0(r,g)$. In particular, we construct a $(10,3;5)$-mixed graph of order 50, that results on a $(10,3;5)$-mixed cage.

\section{Monotonicity and connectivity}

This section is divided into two parts. In the first one we focus on the study of the  monotonicity.  In the second part we study the connectivity of mixed cages.

\subsection{Monotonicity}
Let $G$ be a mixed graph. Let $E^*(G)=A(G)\cup E(G)$ and, if $v$ is a vertex of $G$, let $N^*(v)=N(v)\cup N^+(v)\cup N^-(v)$

\begin{lemma} \label{alternantes}
Let $z\in \{1,2\}$. Every $[z,r;g]$-mixed cage contains a cycle of length $g$ with either two consecutive arcs or two consecutive edges.
\end{lemma}
\begin{proof}
Let $G$ be a $[z,r;g]$-mixed cage with $z\in\{1,2\}$.
Suppose that every cycle of $G$ of length $g$ is alternating by arcs and edges. 
Let $C$ be a cycle of length $g$.
Let $\overrightarrow{xu}\in A(C)$ and  $uv\in E(C)$. Since $C$ is an induced cycle, there exists an arc $\overrightarrow{uy}$ with $y\in V(G)\setminus V(C)$.
Let $N(y)=\{v_1,\ldots,v_r\}$ and $u'\in N^+(y)$.
If $w\in N(y)\cup N^+(y)$, then $w\not\in V(C)$. Otherwise, $G$ would contains a cycle of length $g$ with either two consecutive arcs or with two consecutive edges,  a contradiction.

We divide the proof into two cases. In each case we construct a $[z,r;g]$-mixed graph with less vertices than $G$, giving a contradiction. 

\emph{Case 1)} Suppose $z=1$.
 If  $r$ is even,
let $E'=\{v_{2i-1}v_{2i}:1 \leq i \leq r/2 \}  \cup \{\overrightarrow{uu'}\}$.
We define $G'$ as $G'=G-y+E'$ (see Figure \ref{mono_31}). 
Observe that $G'$ is a mixed graph 1-regular in arcs and $r$-regular in edges. Moreover, the cycle $C$ is totally contained in $G'$. Hence, $G'$ is a $[1,r;g']$-mixed graph with $g'=g(G') \le g$.
Let $C'$ be a cycle of $ G' $ such that $\ell(C')=g'$. 

\begin{figure}[h]
\centering
\includegraphics[width=1\linewidth]{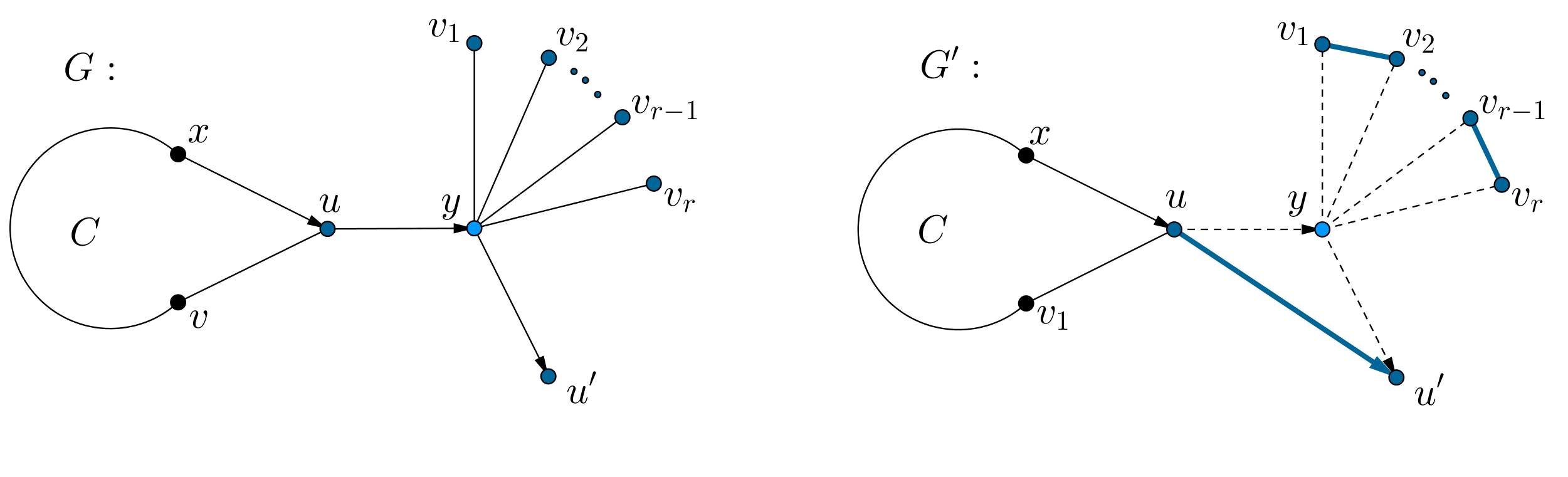}
\caption{Construction of $G'$ from a $[1,r;g]$-mixed cage with $r$ even.}
\label{mono_31}
\end{figure}

If $|E^*(C')\cap E'|=0$, then $C'$ is totally contained in $G$. Thus, $g\le \ell(C)=g'$, a contradiction.
Suppose that $|E^*(C')\cap E'|=1$. Let $\{\alpha_i\alpha_j\}=E^*(C')\cap E'$. Observe that $C'-\alpha_i\alpha_j$ is totally contained in $G$. Hence, either  $\ell(C')\geq d_{G-y}(\alpha_i,\alpha_j)+1\geq g$  or $\ell(C')\geq d_{G-y}(\alpha_j,\alpha_i)+1\geq g$, giving a contradiction.
Continue assuming that $|E(C')\cap E'|\geq 2$. Since $E'$ is an independent set of arcs and edges, there exist $e_1,e_2\in E^*(C)\cap E'$ such that $C'$ contains an $\alpha_i\alpha_j$-path totally contained in $G-y$. Therefore, $\ell(C') \geq d_{G-y}(\alpha_i,\alpha_j)+2 > g$, a contradiction.

If  $r$ is odd, let $w=v_r$,  $N(w)=\{v'_1,\ldots,v'_r\}$, where $v'_r=y$, $N^-(w)=\{x'\}$ and $N^+(w)=\{y'\}$. Let $E'=\{v_{2i-1}v_{2i},v'_{2i-1}v'_{2i}:1 \leq i \leq (r-1)/{2}\}\cup\{\overrightarrow{uu'},\overrightarrow{x'y'}\}$. 
Define the mixed graph $G'$ as $G'=G-\{y,w\}+E'$.
By a  similar analysis to the previous case, we conclude that $g(G)=g$.
Therefore,  $G'$ is a $[1,r;g]$-mixed graph with two vertices less than $G$, yielding a contradiction.

\emph{Case 2)} Suppose $z=2$.
Let  $N^-(y)=\{u,s\}$ and $N^+(y)=\{u',s'\}$.
If $r$ is even,
let $E'=\{v_1v_2,v_3v_4,\dots,$ $v_{r-1}v_r\}$. 
Since $d^-(u')=d^-(s')=2$, it follows that $|N^{-}(u')\cap \{u,s\}|\le 1$ and $|N^{-}(s')\cap \{u,s\}|\le 1$. 
Next, we define a set  $A'$ depending on the  the sets $N^-(u')$ and $N^-(s')$. If  $u\in N^-(u')$ or $s\in N^-(s')$, then $A'=\{\overrightarrow{us'},\overrightarrow{su'}\}$. In other case, $A'=\{\overrightarrow{uu'},\overrightarrow{ss'}\}$.
Define $G'$ as $G'=G-y+E'+A'$ (see Figure \ref{mono2_3}). By a similar analysis to that of \emph{Case 1)}, $G'$ is a $[2,r;g]$-mixed graph with less order than $G$, a contradiction.

\begin{figure}[h]
\centering
\includegraphics[width=0.7\linewidth]{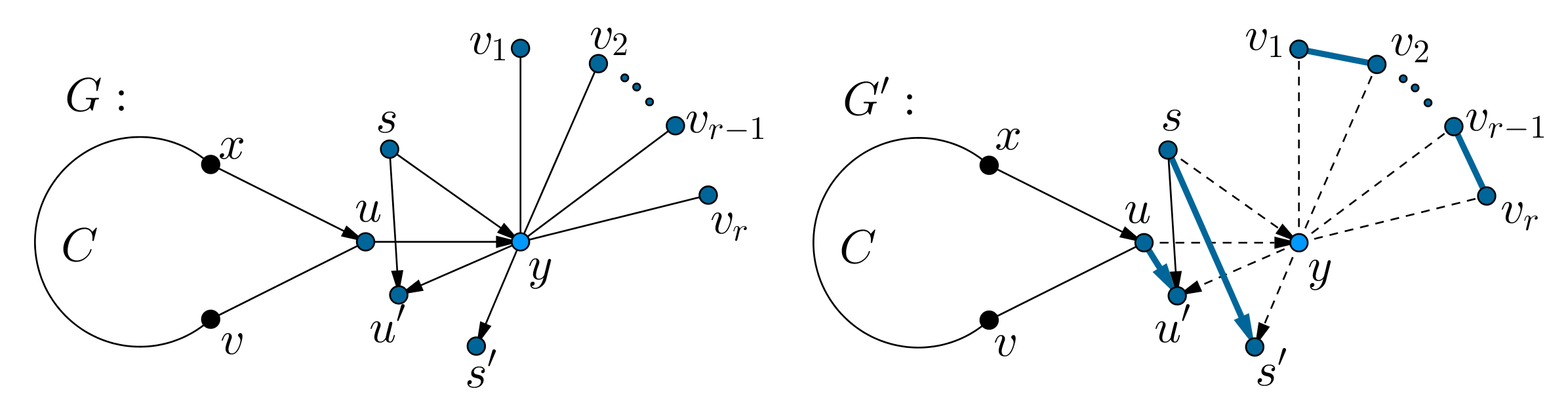}
\caption{Construction of $G'$ from a $[2,r;g]$-mixed cage with $r$ even and $\protect\overrightarrow{su'}\in A(G)$.}
\label{mono2_3}
\end{figure}

If $r$ is odd, let $w=v_r$, $N(w)=\{v'_1,\ldots,v'_r\}$, where $v'_r=y$, $N^-(w)=\{x',x''\}$ and $N^+(w)=\{y',y''\}$.
Let $E'=\{v_{2i-1}v_{2i},v'_{2i-1}v'_{2i}:1 \leq i \leq (r-1)/{2}\}$. 
Since $d^-(u')=d^-(s')=2$, it follows that $|N^{-}(u')\cap \{u,s\}|\le 1$ and $|N^{-}(s')\cap \{u,s\}|\le 1$. Define a set $A_y$ depending on the sets $N^-(u')$ and $N^-(s')$.
If $u\in N^-(u')$ or $s\in N^-(s')$, then $A_y=\{\overrightarrow{su'},\overrightarrow{us'}\}$. In other case,  $A_y=\{\overrightarrow{uu'},\overrightarrow{ss'}\}$. 
Analogously, we define  $A_w$ depending on the sets $N^-(y')$ and $N^-(y'')$. If  $x'\in N^-(y')$ or $x''\in N^-(y'')$, then $A_w=\{\overrightarrow{x'y''},\overrightarrow{x''y'}\}$. In other case, $A_w=\{\overrightarrow{x'y'},\overrightarrow{x''y''}\}$.

Define $G'$ as $G'=G-\{y,w\}+E'+A_y+A_w$. Again, by a similar analysis to that of \emph{Case 1)}, a contradiction is obtained.

Therefore, every $[z,r;g]$-mixed cage have at less one cycle of length $g$ with two consecutive arcs or two  consecutive edges.
\end{proof}

In \cite{AHM19}, Araujo-Pardo, Hern{\' a}ndez-Cruz and Montellano-Ballesteros, calculated a general lower bound for a $[z,r;g]$-mixed cage in the following theorem:

\begin{theorem}\label{TeoCotaInf}
If $n[1,r;g]$ is the order of a $[1,r;g]$-mixed cage, then

$n[1,r;g]\geq n_0[1,r;g] = \left\lbrace
\begin{array}{ll}
2\left(1 + \sum_{i=1}^{(g-3)/2} n_0(r,2i+1)\right) + n_0(r,g) & \textup{if $g$ is odd;}\\
2\left(1 + \sum_{i=1}^{(g-2)/2} n_0(r,2i+1)\right) & \textup{if $g$ is even.}
\end{array}
\right.$
\end{theorem}

Now, we can prove the main theorem of this section.

\begin{theorem} \label{monotonia1}
Let $z\in \{1,2\}$, $r \geq 1$ and $3\leq g_1 < g_2$ be integers, then $$
n[z,r;g_1]<n[z,r;g_2].
$$
\end{theorem}

\begin{proof}
It suffices to show that if $z\in \{1,2\}$, $r \geq 1$ and $g \geq 3$, then $n[z,r;g]<n[z,r;g+1]$.
Let $G$ be a $[z,r;g+1]$-mixed cage.
Let $C$ be a cycle of $G$ such that $\ell(C)=g+1$.
By Lemma \ref{alternantes}, $C$ contains two consecutive arcs or two consecutive edges.
Let $u\in V(C)$. Suppose that  $N(u)=\{v_1,\ldots,v_r\}$, $x_1 \in N^-(u)$ and $y_1\in N^+(u)$.

We divide the proof in cases depending on the value of  $z$ and the parity of $r$.
However, the general reasoning for all cases is the same: from the graph $G$, by deleting a set of vertices and adding a set of arcs and a set of edges, a $[z,r;g']$-mixed graph $G'$ with girth $g'<g+1$ and $|V(G')|< n[z,r;g]$ is constructed. 

\emph{Case 1)} Suppose $z=1$.
If $g=3$, by Theorem \ref{TeoCotaInf},  $n[1,r;3]=2+n_0(r,3)=r+3$, $n[1,r;4]\geq 2(1+n_0(r,3))=2r+4$, and the result follows. Continue assuming $g\ge 4$.

\emph{Case 1.1)} Suppose that $r$ is even.
If $C$ has two consecutive edges $v_1u$ and $uv_2$ (see Figure \ref{mono_1}), let $E'=\{v_{2i-1}v_{2i}:1 \leq i \leq r/2 \}\cup \{\overrightarrow{x_1y_1}\}$.
\begin{figure}[h]
\centering
\includegraphics[width=1\linewidth]{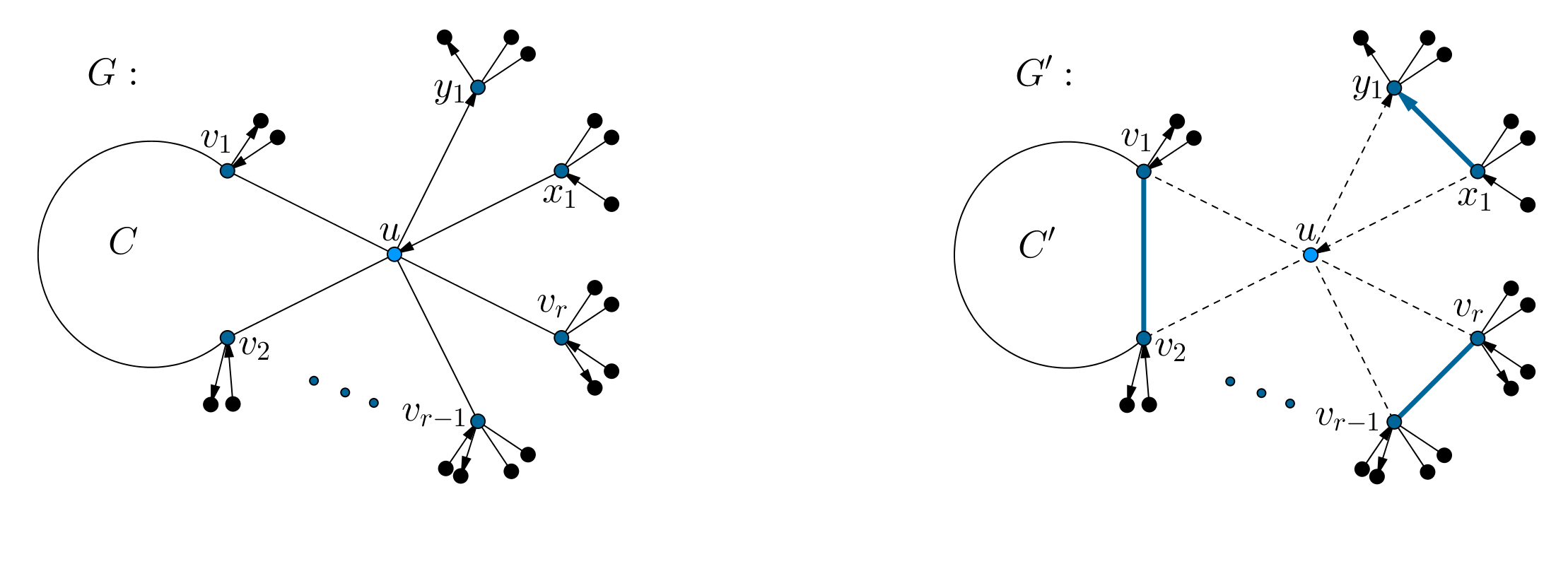}
\caption{Operation in a $[1,r;g+1]$-mixed cage with $r$ even, in a cycle with two edges consecutive.}
\label{mono_1}
\end{figure}

Let  $G'=G-u+E'$. Observe that $g(G')\leq g$, since $G'$ contains the cycle $C-u+v_1v_2$. We claim that $g(G')= g$. Let $C'$ be a cycle of $G'$ such that $\ell(C')=g(G')\le g$. If $E(C')\cap E'=\emptyset$, then $E^*(C')\subseteq E(G)$, implying  that $\ell(C')\geq g+1$, a contradiction. Hence, $E^*(C')\cap E'\neq\emptyset$.
If $|E^*(C')\cap E'|=1$, then $E^*(C')\cap E'=\{\alpha_i\alpha_j\}$, which implies that $C'- \alpha_i\alpha_j$ is an $\alpha_i\alpha_j$-path or an $\alpha_j\alpha_i$-path totally contained in $G$ of length at least $g-1$. Therefore, $g\le d_{G-u}(\alpha_i,\alpha_j)+1 \le \ell(C')=g(G') \leq g$.
Suppose that  $|E^* (C')\cap E'|\geq 2$. Since $E'$ is an independent set of edges and arcs, there exist $e_1,e_2\in E^*(C')\cap E'$ such that  there is an $\alpha_i\alpha_j$-path in $C'$ totally contained in $G-u$, where $\alpha_i$ is a vertex of $e_1$ and $\alpha_j$ is a vertex of $e_2$. Since $\alpha_i,\alpha_j \in N^*(u)$, the length of  every $\alpha_i\alpha_j$-path in $G-u$ is at least $g-1$. Hence, 
$g<d_{G-u}(\alpha_i,\alpha_j)+2 \leq \ell(C')\le g$, a contradiction.

Therefore, $g(G')=g$ and $G'$ is a $[1,r;g]$-mixed graph. Thus,
$$
n[1,r;g]\le |V(G')|=|V(G)|-1<n[1,r;g+1],
$$
and the result follows.

The case in which $C$ contains two consecutive arcs is analogous. 

\emph{Case 1.2)} Suppose that $r$ is odd. 
If $v_1u,uv_2\in E(C)$
(see Figure \ref{mono_3}),
let $w=v_r$, $N(w)=\{v'_1,\ldots,v'_r\}$, where $v'_r=u$, $N^-(w)=\{x'\}$ and $N^+(w)=\{y'\}$. Let $E'=\{v_{2i-1}v_{2i},v'_{2i-1}v'_{2i}:1 \leq i \leq (r-1)/2\}\cup\{\overrightarrow{x_1y_1},\overrightarrow{x'y'}\}$. 

Let $G'=G-\{u,w\}+E'$. Notice that $G'$ contains a cycle of length $g$, therefore  $g(G')\le g$. Let $C'$ be a cycle of $G'$ such that $\ell(C')= g(G')$. 
By a similar analysis to the \emph{Case 1.1)}, it follows that $\ell(C')\geq d_{G-\{u,w\}}(\alpha_i,\alpha_j)+|E^*(C')\cap E'|\geq g$, where $\alpha_i$ and $\alpha_j$ are vertices of the edges or arcs in $E^*(C')\cap E'$.

Therefore, $G'$ is a $[1,r;g]$-mixed graph with $n[1,r;g+1]-2$ vertices. Thus 
$$
n[1,r;g]\le |V(G')|=|V(G)|-1<n[1,r;g+1].
$$
The case in which $C$ contains two consecutive arcs is analogous.

\begin{figure}[h]
\centering
\includegraphics[width=0.7\linewidth]{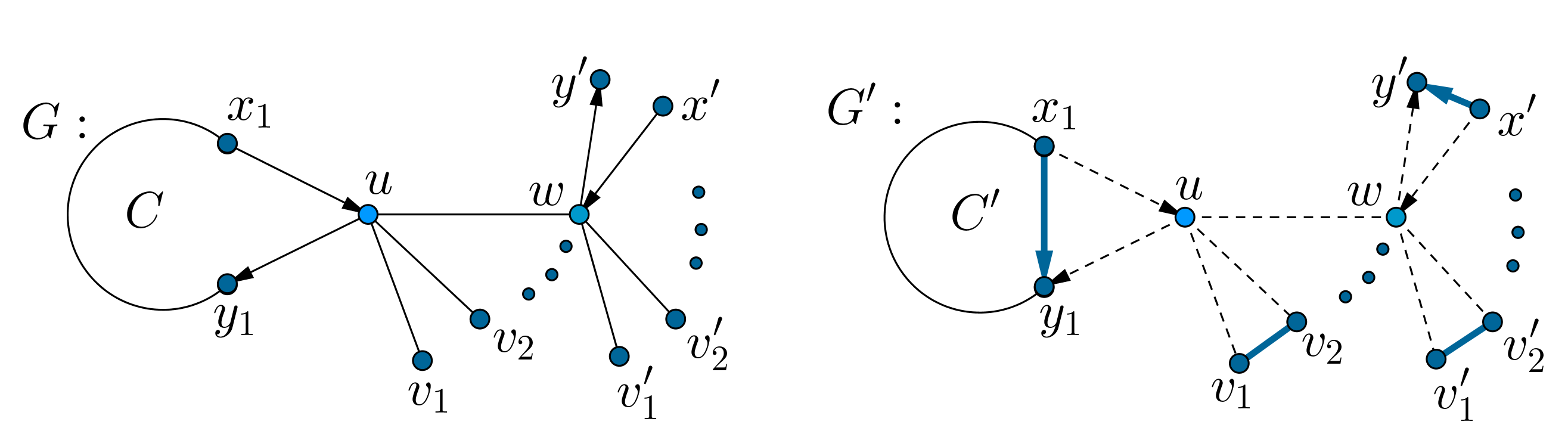}
\caption{Operation in a $[1,r;g+1]$-mixed cage with $r$ odd, in a cycle with two arcs consecutive.}
\label{mono_3}
\end{figure}

\emph{Case 2)} Suppose $z=2$.
Let $N^-(u)=\{x_1, x_2\}$ and $N^+(u)=\{y_1, y_2\}$.

\emph{Case 2.1)} Suppose that $r$ is even.
If $v_1u,uv_2\in E(C)$,  let $E'=\{v_{2i-1}v_{2i}:1\leq i \leq r/2\}$. 
Next, we define a set $A'$ depending on the sets $N^-(y_1)$ and $N^-(y_2)$.
If $x_1\in N^-(y_1)$ or $x_2\in N^-(y_2)$, then $A'=\{\overrightarrow{x_1y_2},\overrightarrow{x_2y_1}\}$. In other case,   $A'=\{\overrightarrow{x_1y_1},\overrightarrow{x_2y_2}\}$.
Let $G'=G-u+E'+A'$ (see Figure \ref{mono2_1}). 
Proceeding as in  \emph{Case 1.1)}, it follows that  $g(G')=g$ and the result follows.

The case in which $\overrightarrow{x_1u},\overrightarrow{uy_1}\in A(C)$ and $y_2 \not\in N^+(x_2)$  is proved in a similar way.

\begin{figure}[h]
\centering
\includegraphics[width=1\linewidth]{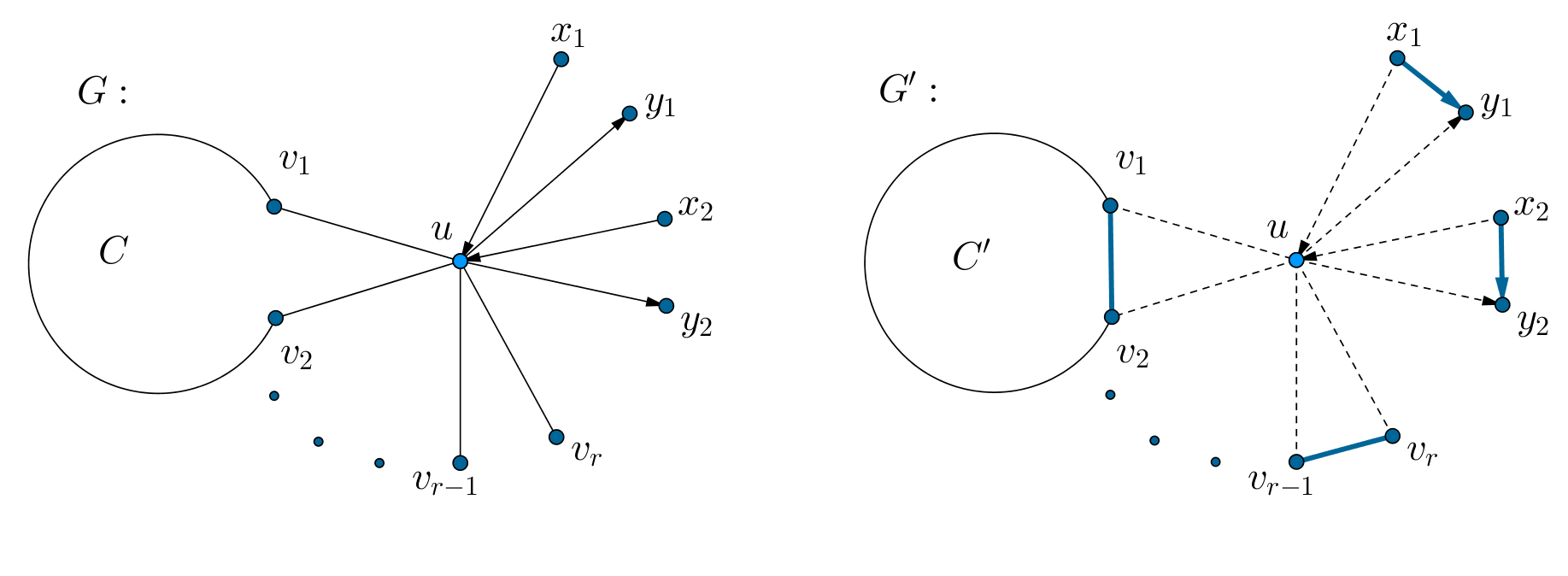}
\caption{Operation in a $[2,r;g+1]$-mixed cage with $r$ even, in a cycle with two edges consecutive and there is no the arc $\protect\overrightarrow{x_1y_1}$ or $\protect\overrightarrow{x_2y_2}$.}
\label{mono2_1}
\end{figure}
Suppose now that $\protect\overrightarrow{x_1u},\protect\overrightarrow{uy_1}\in A(C)$ and $y_2 \in N^+(x_2)$.
If $y'\in N^+(y_2)\cap V(C)$, let $E'=\{v_{2i-1}v_{2i}:1\leq i \leq r/2\}\cup \{\protect\overrightarrow{x_1y_2},\protect\overrightarrow{x_2y_1}\}$.
Let $G'=G-u+E'$ (see Figure \ref{mono2_2}). We claim that $g(G')= g$. Since $G'$ contains the cycle  $C-u+\protect\overrightarrow{x_1y_2}+\protect\overrightarrow{y_2y'}$,  it follows that $g(G')\leq g+1$. If $g(G')= g+1$, then $G'$ is a $[2,r;g+1]$-mixed graph with $n[2,r;g+1]-1$ vertices, a contradiction. Therefore $g(G')\le g$ and by a similar analysis to that in \emph{Case 1.1)}, it follows that $n[2,r;g]\leq|V(G')|<|V(G)|$.

\begin{figure}[h]
\centering
\includegraphics[width=1\linewidth]{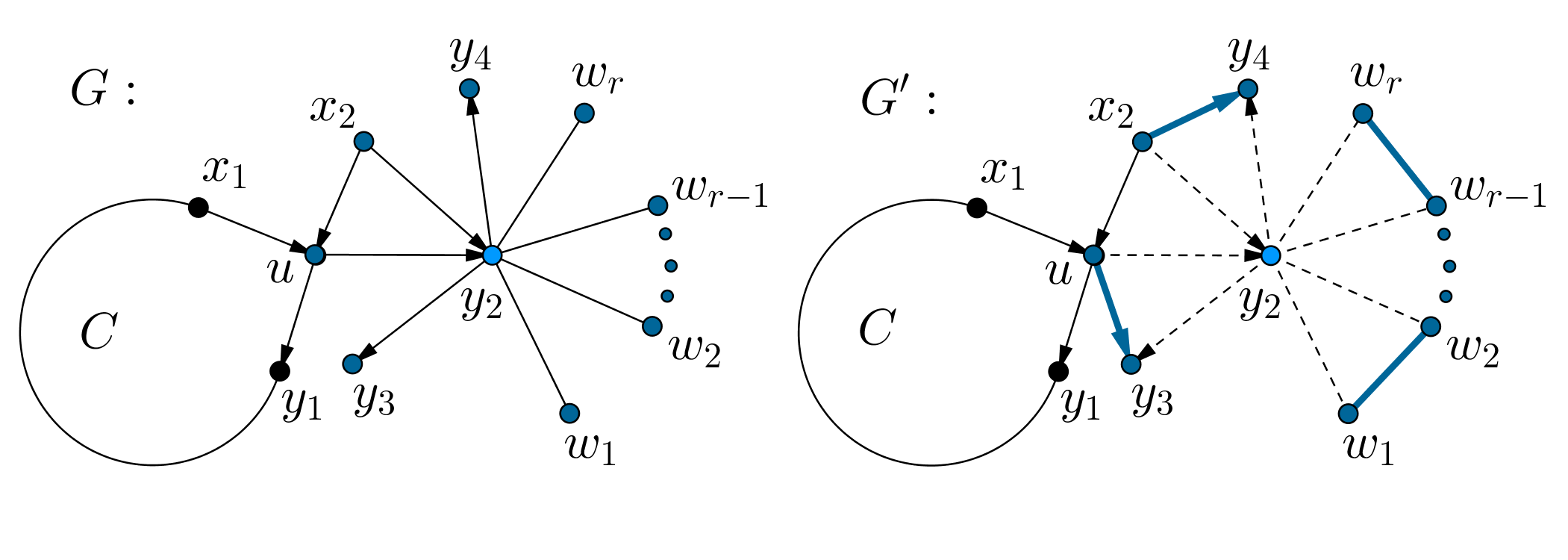}
\caption{Operation in a $[2,r;g+1]$-mixed cage with $r$ even, in a cycle with two arcs consecutive and there is no the arc $\protect\overrightarrow{x_2y_2}$.}
\label{mono2_2}
\end{figure}

If $N^+(y_2)\cap V(C)=\emptyset$, let $N^+(y_2)=\{y_3,y_4\}$, $N(y_2)=\{w_1,\ldots,w_r\}$. Set $E'=\{w_1w_2,w_3w_4,\ldots,w_{r-1}w_r\}\cup \{\protect\overrightarrow{uy_3},\overrightarrow{x_2y_4}\}$. Let $G'=G-y_2+E'$, note that $G'$ is a $[2,r,g']$-mixed graph, with one vertex less than $G$ and $g'=g(G')<g+1$, since the cycle $C$ is contained in $G'$.
Let $C'$ be a cycle such that $\ell(C')=g'$, similarly to  \emph{Case 1.1)},  we conclude that $\ell(C')\geq g$. Thus, $g'=g$ and $G'$ is a $[2,r;g]$-mixed graph with $n[2,r;g+1]-1$ vertices and $n[2,r;g]\le |V(G')|<n[2,r;g+1]$.

\emph{Case 2.2)} Suppose that $r$ is odd.
Let $s=v_r$, $N(s)=\{v'_1,\ldots,v'_r\}$, where $v'_r=u$, $N^-(s)=\{x'_1,x'_2\}$ and $N^+(s)=\{y'_1,y'_2\}$.

 Suppose that $v_1u,uv_2\in E(C)$.
Let $E'=\{v_{2i-1}v_{2i},v'_{2i-1}v'_{2i}:1 \leq i \leq (r-1)/2\}\cup A_u\cup A_s$, where $A_u$ and $A_s$ are defined  depending on the sets  $N^-(y_i)$ and $N^-(y'_i)$ for $i \in \{1,2\}$. Since $d^-(y_i)=2$, it follows that  $|N^-(y_i)\cap \{x_1,x_2\}|\le 1$. 
If either $x_1\in N^-(y_1)$ or $x_2\in N^-(y_2)$, then  $A_u=\{\overrightarrow{x_1y_2},\overrightarrow{x_2y_1}\}$. 
In other case set  $A_u=\{\overrightarrow{x_1y_1},\overrightarrow{x_2y_2}\}$. 
If either $x'_1 \in N^-(y'_1)$ or $x'_2 \in N^-(y'_2)$, then $A_s=\{\overrightarrow{x'_1y'_2},\overrightarrow{x'_2y'_1}\}$. 
In other case  $A_s=\{\overrightarrow{x'_1y'_1},\overrightarrow{x'_1y'_1}\}$.
Let $G'=G-\{u,s\}+E'$. Proceeding as in  \emph{Case 1.1)}, it follows  that $g(G')= g$, and the result follows.

The case in which $\overrightarrow{x_1u},\overrightarrow{uy_1}\in A(C)$ and $y_2 \not\in N^+(x_2)$ is proved in a similar way.

Next, suppose that $\overrightarrow{x_1u},\overrightarrow{uy_1}\in A(C)$, and $y_2\in N^+(x_2)$. 
If $y'\in N^+(y_2)\cap V(C)$, let $E'=\{v_{2i-1}v_{2i},v'_{2i-1}v'_{2i}:  1 \leq i \leq (r-1)/{2}\} \cup \{\overrightarrow{x_1y_2},\overrightarrow{x_2y_1}\}\cup A_s$. If either $x'_1 \in N^-(y'_1)$ or $x'_2 \in N^-(y'_2)$, then  $A_s=\{\overrightarrow{x'_1y'_2},\overrightarrow{x'_2y'_1}\}$. In other case,  $A_s=\{\overrightarrow{x'_1y'_1},\overrightarrow{x'_2y'_2}\}$.
Let $G'=G-\{u,s\}+E'$ (see Figure \ref{mono2_4}). Since $G'$ contains the cycle $C-u+\overrightarrow{x_1y_2}+\overrightarrow{y_2y'}$, therefore $g(G')\leq g+1$. If $g(G')= g+1$, it follows that $G'$ is a $[2,r;g+1]$-mixed graph, a contradiction. Hence, $g(G')\le g$ and proceeding as in \emph{Case 1.1)}, it follows that $n[2,r;g]\leq|V(G')|<|V(G)|$.

\begin{figure}[h]
\centering
\includegraphics[width=0.8\linewidth]{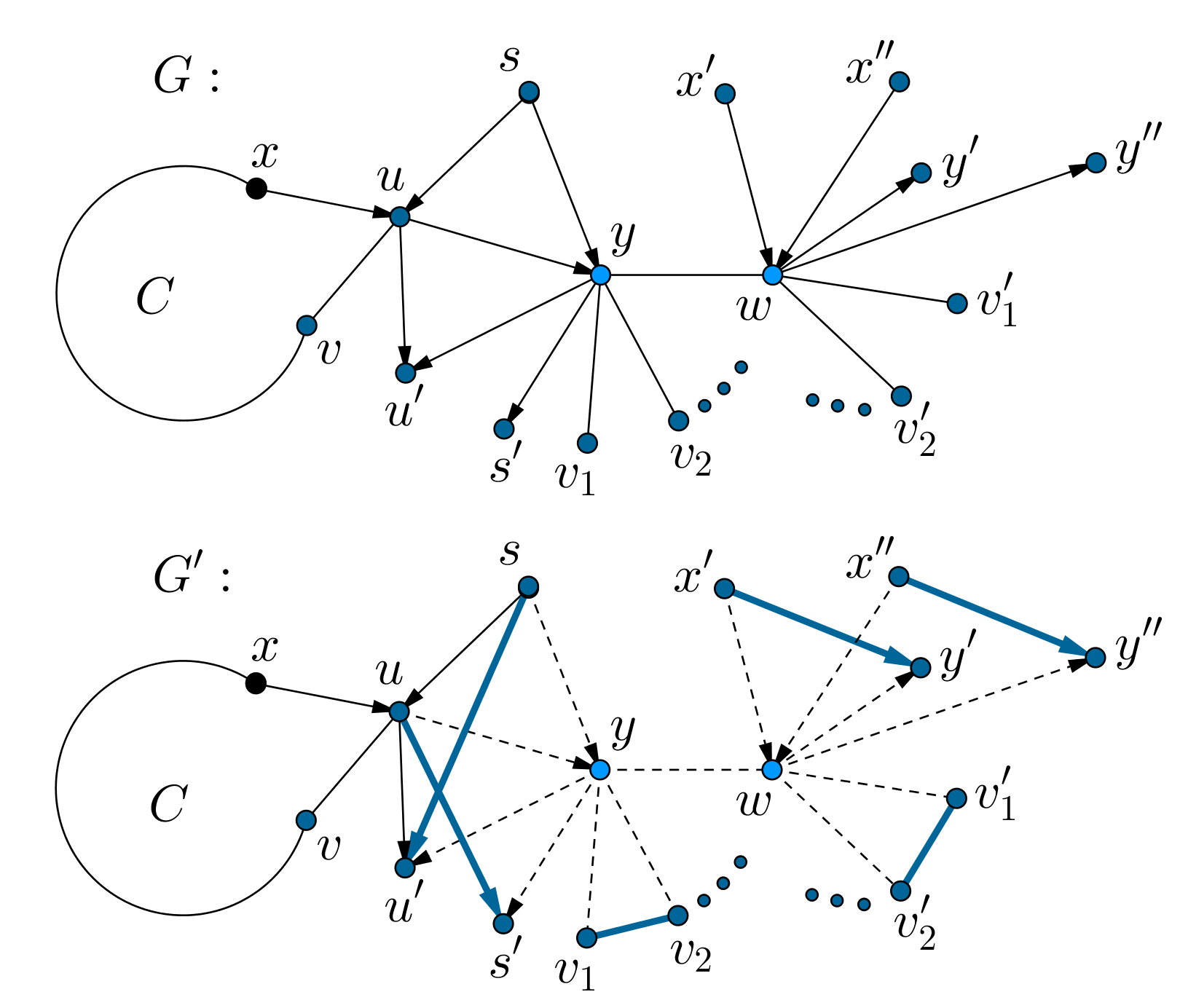}
\caption{Operation in a $[2,r;g+1]$-mixed cage with $r$ odd, in a cycle with two consecutive arcs  and there is the arc $\protect\overrightarrow{x_2y_2}$.}
\label{mono2_4}
\end{figure}

If $N^+(y_2)\cap V(C)=\emptyset$, let $N^+(y_2)=\{y_3,y_4\}$, $N(y_2)=\{w_1,\ldots,w_r\}$, with $w_r=t$. Let $N(t)=\{w'_1,\ldots,w'_r\}$, with $w'_r=y_2$, $N^-(t)=\{x'_1,x'_2\}$ and $N^+(t)=\{y'_1,y'_2\}$. 
Let $E'=\{w_{2i-1}w_{2i},w'_{2i-1}w'_{2i}: 1 \leq i \leq (r-1)/{2}\}\cup\{\overrightarrow{uy_3},\overrightarrow{x_2y_4}\}\cup A_t$, where  $A_t=\{\overrightarrow{x'_1y'_2},\overrightarrow{x'_2y'_1}\}$ if
either $x'_1 \in N^-(y'_1)$ or $x'_2 \in N^-(y'_2)$, and $A_t=\{\overrightarrow{x'_1y'_1},\overrightarrow{x'_2y'_2}\}$ 
in any  other case.

Let $G'=G-\{y_2,t\}+E'$. 
Since $C$ is contained in $G'$, it follows that $g(G')<g+1$. Let $C'$ a cycle such that $\ell(C')=g(G')$, proceeding as in \emph{Case 1.1)}, it can be concluded that $g(G')=g$. Hence, $G'$ is a $[2,r;g]$-mixed graph with $n[2,r;g+1]-2$ vertices. Therefore $n[2,r;g]\le |V(G')|<n[2,r;g+1]$, and the theorem is proved.

\end{proof}

\subsection{Connectivity of a mixed cage}

In this subsection we give some results on the connectivity  of a mixed cage.
A mixed graph $G$ is \emph{strong} if for every two vertices $u$ and $v$ of $G$ there exists a $uv$-path and a $vu$-path.
Clearly, if $G$ is a  $[z,r;g]$-mixed cage, then the underlying graph of $G$ is connected. 

\begin{theorem}
If $G$ is a $[z,r;g]$-mixed cage, then $G$ is strongly connected. 
\end{theorem}
\begin{proof}
Let $G$ be   a $[z,r;g]$-mixed cage. Suppose to the contrary that $G$ is not strong. Let $H_1,\dots ,H_k$ be the strong components of $G$. Note that there are no edges between the strong components of $G$. Let $H^*=\cup_{i=2}^k H_i$. Since  the underlying graph of $G$ is connected, there is at least one arc between $H_1$ and $H^*$. Furthermore, all the arcs between $H_1$ and $H^*$ have the same direction. Suppose without lose of generality that $[V(H_1),V(H^*)]\neq \emptyset$. 
The number of arcs of $G$ is 
$$
|A(G)|=|V(G)|z=(|V(H_1)|+|V(H^*)|)z.
$$ 
On the other hand, for every vertex $v\in V(H_1)$ it follows that $d^-(v)=z$, and for every vertex $u\in V(H^*)$, $d^+(u)=z$. Hence, $|A(H_1)|=|V(H_1)|z$ and $|A(H^*)|=|V(H^*)|z$. Therefore, 
$|A(G)|=|A(H_1)|+|A(H^*)|+|[V(H_1),V(H^*)]|=|V(H)|z+|V(H^*)|z+|[V(H_1),V(H^*)]|$,  
implying that $|[V(H_1),V(H^*)]|=0$,  a contradiction.
\end{proof}

\begin{theorem}
The underlying graph of a $[z,r;g]$-mixed cage is 2-connected, for $z\in \{1,2\}$.
\end{theorem}
\begin{proof}
Let $G$ be a $[z,r;g]$-mixed cage, $z\in \{1,2\}$. Suppose that there exists a vertex $v\in V(G)$ such that the underlying graph of $G-v$ is not connected. Let $H$ be a connected component of $G-v$ of minimum order. Observe that $|V(H)|<|V(G)|/2$. 
Let $\overleftarrow{H}$ be the reverse graph of $H$ and let $u'$ denote the corresponding vertex of $u$ in $\overleftarrow{H}$. We construct a new graph $G^*$ formed by the disjoint union of $H$ and $\overleftarrow{H}$, an edge set $E'=\{uu':u\in V(H)\cap N(v), u'\in V(\overleftarrow{H})\}$  and an arc set $A'=\{\overrightarrow{uu'}: u \in V(H)\cap N^-(v), u'\in V(\overleftarrow{H})\}\cup  \{\overrightarrow{u'u}: u \in V(H)\cap N^+(v), u'\in V(\overleftarrow{H})\}$.
Observe that $G^*$ is a $[z,r,g(G^*)]$-mixed graph. Let $C$ be a cycle of  length  $g(G^*)$. 
Since $|V(G^*)|=2|V(H)|<|V(G)|$, by Theorem \ref{monotonia1}, $g(G^*)<g$. Hence $(E(C)\cup A(C))\cap (E'\cup A')\neq \emptyset$. Thus, there exists at least two vertices $u_1$ and $u_2$ of $H$ which are the endings of those edges or arcs belonging to $C$. Notice that  $u_1$ and $u_2$ are at distance at least $g-2$ in $G-v$. Hence, $g(G^*)\ge 2(g-2)>g$. Therefore, $G^*$ is a $[z,r;g^*]$-mixed cage with $g^*\geq g$. By Theorem \ref{monotonia1}, 
$n[z,r;g']\leq |V(G^*)| < |V(G)| = n[z,r;g]$, a contradiction.
\end{proof}

\section{Construction of mixed graphs}

In this section some constructions of families of mixed graphs are presented. 

\subsection{Lower bounds}

In this subsection we give a lower bound for $n[z,r;g]$. Let $G$ be a mixed graph. Given a vertex $u$ of $G$, we define the \emph{projection} of $u$ as $\overrightarrow{N}(u)=N^+(u)\cup N(u)$. Similarly, the {\em injection} of $u$ is the set $\overleftarrow{N}(u)=N^-(u)\cup N(u)$.

\begin{proposition}\label{lowerbound1}
The order of a $[z,r;g]$-mixed cage is at least $n_0(r,g)+2z$.
\end{proposition}
\begin{proof}
Let $G$ be a $[z,r;g]$-mixed cage. By deleting the arcs of $G$ we obtain an $(r,g')$-graph with $g'\ge g$. Hence by the Moore bound and the monotonocity it follows that $|V(G)|\ge n_0(r,g)$. In addition, since every vertex of $G$ has $z$ ex-neighbors and $z$ in-neighbors, it follows that  $|V(G)\ge n_0(r,g)+2z$.

\end{proof}
Next we improve the previous lower bound for some specific parameters.

\begin{theorem}\label{cota_cuello_5}
The order of a $[10,3;5]$-mixed cage is at least 50.
\end{theorem}

\begin{proof}
Let $G$ be a $[10,3;5]$-mixed cage.
By Proposition \ref{lowerbound1}, it follows that $|V(G)|\geq 30$.  Let $G'=G-A(G)$. Observe that $G'$ is a $(3,g')$-graph with $g'\ge 5$. Let $u\in V(G)$ and let $N_2(u)$ be the set of vertices of $G'$ at distance  at most 2 from $u$. Observe that $N^+(u)\cap N_2(u)=\emptyset$ and 
$N^-(u)\cap N_2(u)=\emptyset$.  

\textbf{Claim}. \emph{There exists a vertex $v\in N^+(u)$ such that $|\overrightarrow{N}(v)\cap N^+(u)|\le 3.$}

Suppose that for every $v\in N^+(u)$, $|\overrightarrow{N}(v)\cap N^+(u)|\ge 4$.  
Therefore, there exists a vertex $w\in N^+(u)$ such that  $|N(w)\cap N^+(u)|\le 2$. Otherwise the mixed graph  induced by $N^+(u)$ would contains a cycle of length at most 4. 
Let $y\in N^+(u)$ and suppose that $|N(y)\cap N^+(u)|= 2$. Let $x_1$, $x_2\in N(y)\cap N^+(u)$ and let $z\in N^+(y)\cap N^+(u)$. 
Since $|\overrightarrow{N}(z)\cap N^+(u)|\ge 4$ and $G$ has girth 5, it follows that   $x_1,x_2\notin \overrightarrow{N}(z)$. Let $Z=\overrightarrow{N}(z)\cap N^+(u)$ (see Figure \ref{estructura}).
\begin{figure}[h]
\centering
\includegraphics[width=0.5\linewidth]{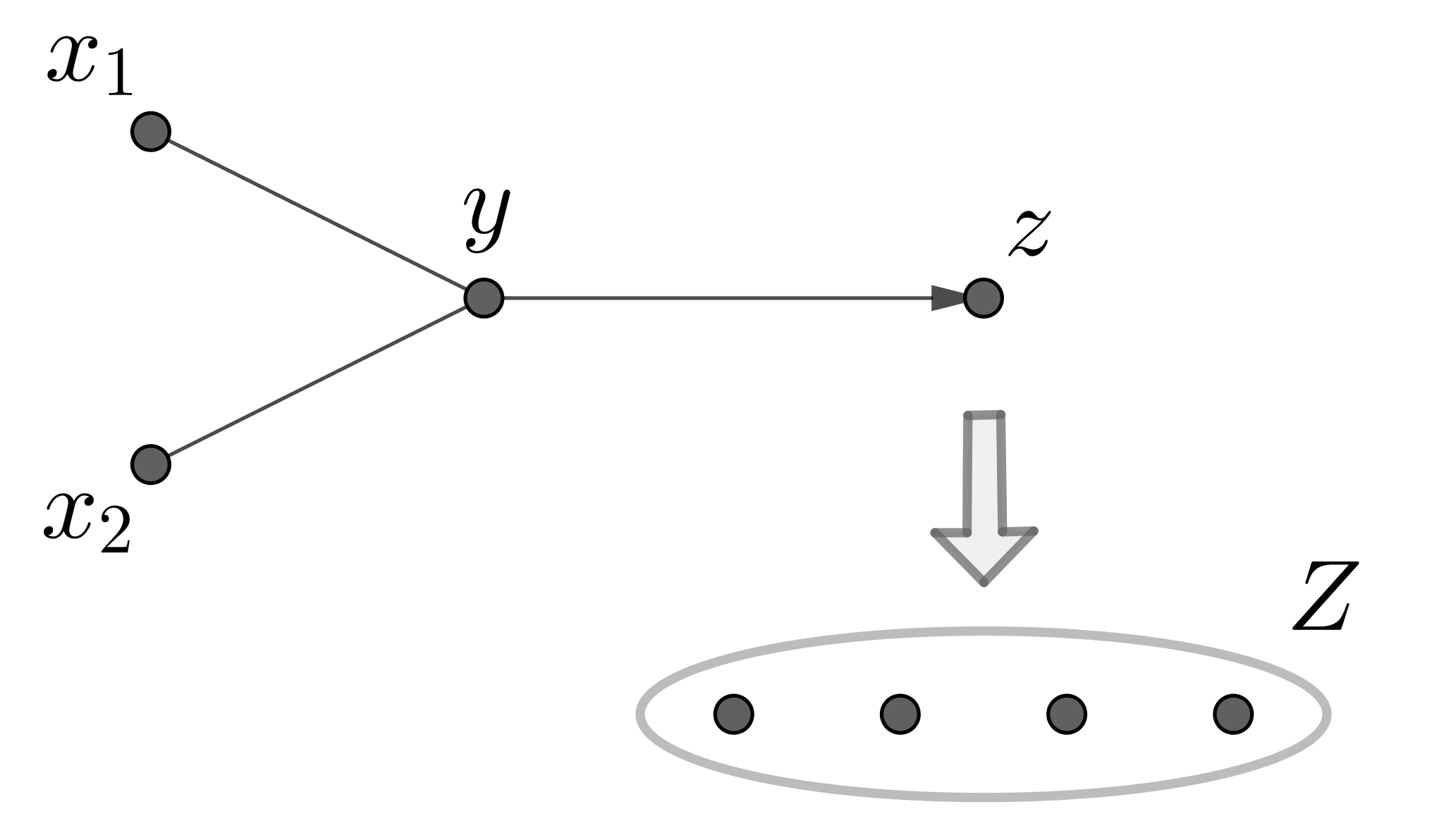}
\caption{Structure with two adjacent edges.}
\label{estructura}
\end{figure}
\vspace{2cm}

Notice that if we want maximize the number of edges and the arcs in $G[Z]$ (the subgraph induced graph by $Z$), there are only nine possibilities for $G[Z]$ (see Figure \ref{9configuraciones}). If there is a vertex $w\in Z$ such that $|\overrightarrow{N}(w)\cap Z|= 1$, then $w$ is incident with an edge  of $G[Z]$ (see Figure \ref{9configuraciones}). Hence, $w$ cannot be adjacent to $z$ with an edge, because the girth of $G$ is $5$. 
Therefore $|\overrightarrow{N}(w)\cap (N^+(u)\setminus Z)|\ge 3$ and since $x_1, x_2,y\notin \overrightarrow{N}(w)$, a contradiction is obtained.

\begin{figure}[h]
\centering
\includegraphics[width=1\linewidth]{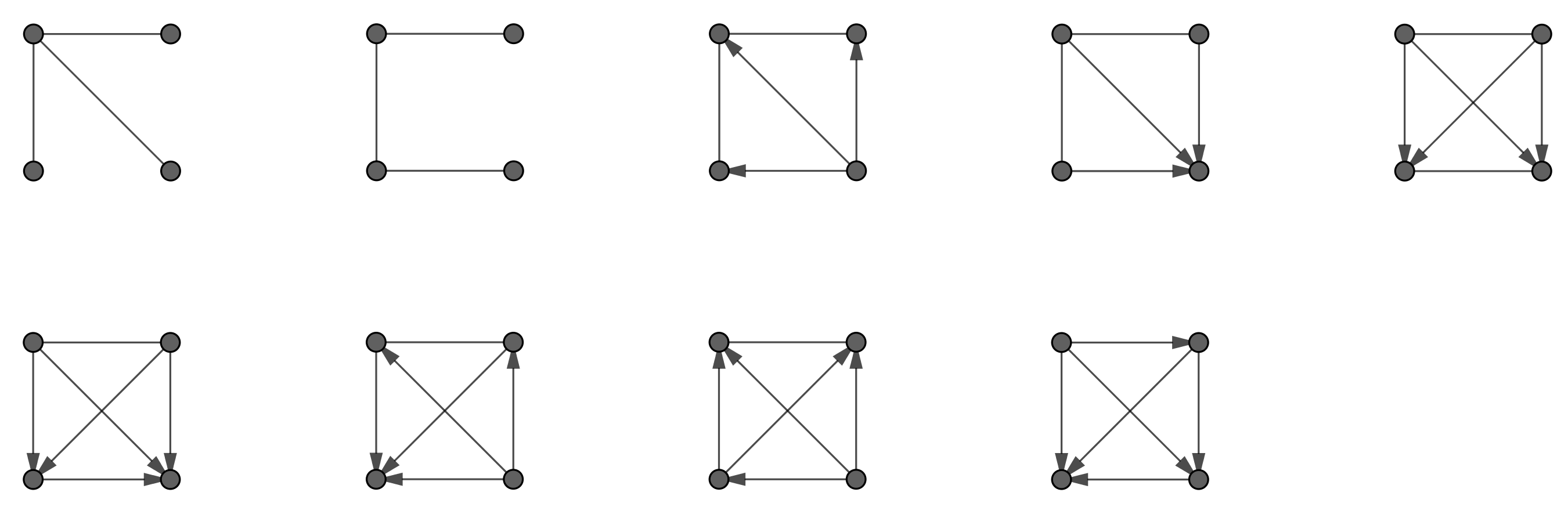}
\caption{The nine configurations of $G[Z]$.}
\label{9configuraciones}
\end{figure}

Similarly, in the possibilities of  $G[Z]$ that have a vertex $w$ with $|\overrightarrow{N}(w)\cap Z|=0$, it follows that 
$|\overrightarrow{N}(w)\cap (N^+(u)\setminus \{z\})|\ge 3$. Since $x_1, x_2,y\notin \overrightarrow{N}(w)$, a contradiction is obtained.
 Consequently, there are no two incident edges in $N^+(u)$.

Let $x\in N(y)\cap N^+(u)$ and $z\in N^+(y) \cap N^+(u)$ (see Figure \ref{estructura_2}).
\begin{figure}[h]
\centering
\includegraphics[width=0.6\linewidth]{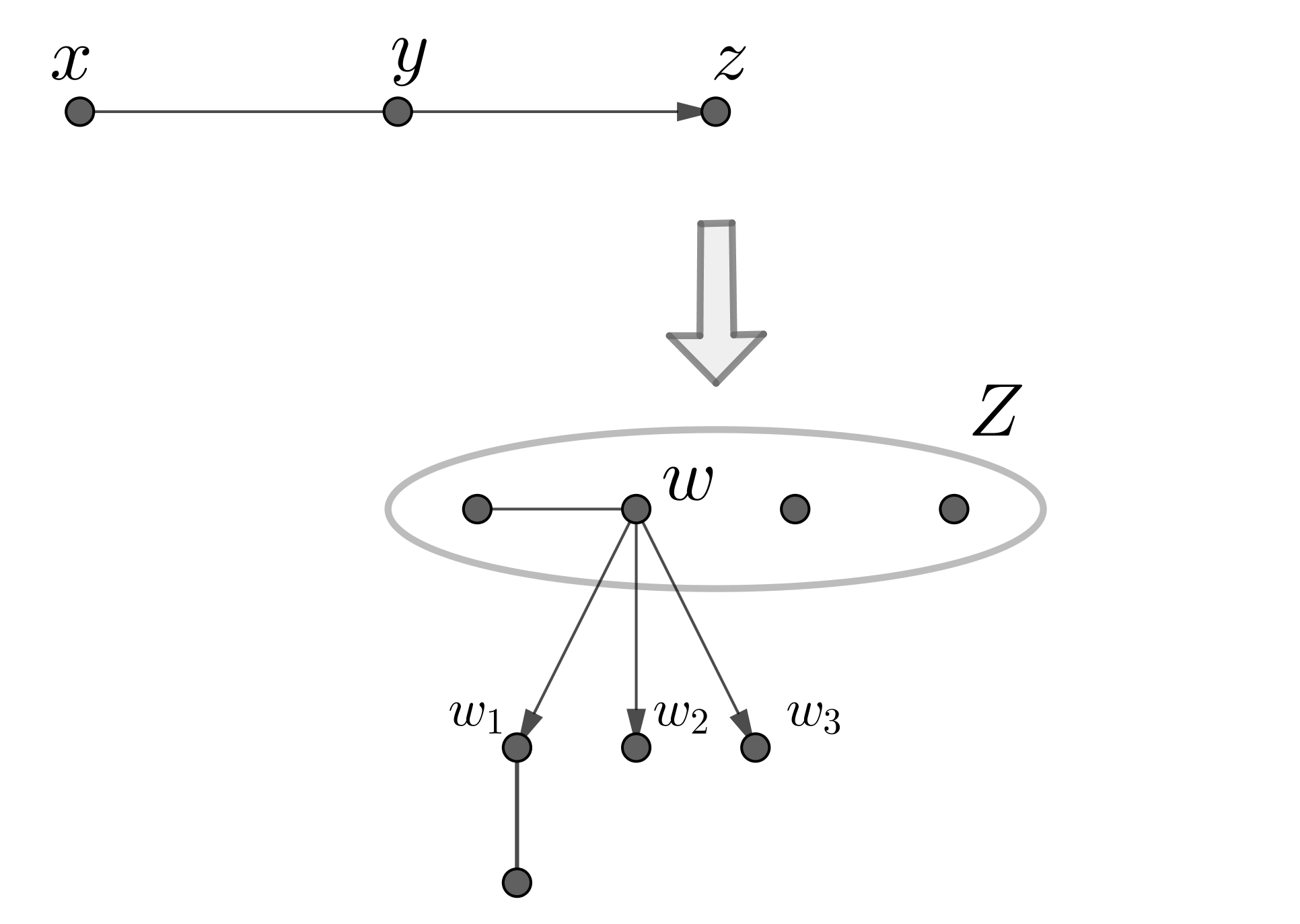}
\caption{Structure with at least 11 vertices in $N^+(u)$.}
\label{estructura_2}
\end{figure}
Observe that  $|\overrightarrow{N}(z)\cap N^+(u)|\ge 4$ and $x,y\notin (\overrightarrow{N}(z)\cap N^+(u))$. Let
 $Z=\overrightarrow{N}(z)\cap N^+(u)$. In this case we only have five possibilities for $G[Z]$ (see  Figure \ref{5configuraciones}).

\begin{figure}[h]
\centering
\includegraphics[width=1\linewidth]{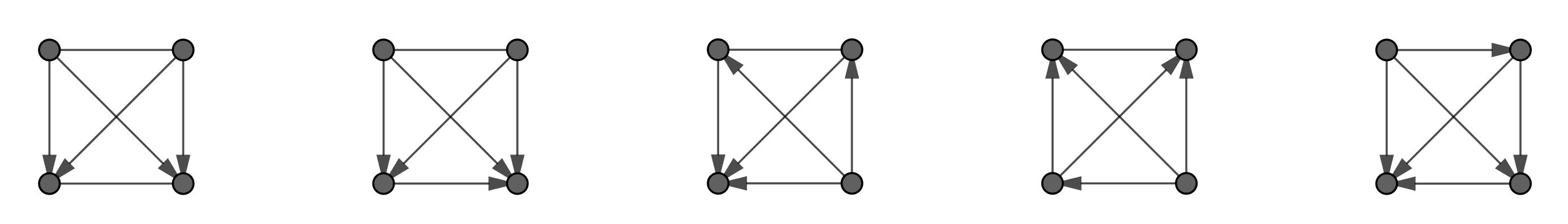}
\caption{The five configurations of four vertices with the maximum number of edges and arcs preserving the five girth and without two adjacent edges.}
\label{5configuraciones}
\end{figure}
Observe that in each one of the possible graphs of $G[Z]$ there is a vertex $w$ with either $|\overrightarrow{N}(w)\cap Z|=0$ or $|\overrightarrow{N}(w)\cap Z|=1$.
Let $w\in Z$ and suppose that $|\overrightarrow{N}(w)\cap Z|=0$. Since $|\overrightarrow{N}(w)\cap N^+(u)|\ge 4$, it follows that  $|N^+(u)|\geq 11$, a contradiction.
If $|\overrightarrow{N}(w)\cap Z|= 1$, then $|N(w)\cap Z|=1$. Moreover, since there are no two incident edges in $N^+(u)$, it follows that $w\in N^+(z)$ and  there are at least three vertices $w_1,w_2,w_3\in (N^+(w)\cap (N^+(u)\setminus\{x,y,z\}))$. Since $|\overrightarrow{N}(w_1)\cap N^+(u)|\ge 4$ and $|N^+(u)|=10$, it follows that $|\overrightarrow{N}(w_1)\cap (Z\cup \{y,z\})|\ge 1$. Therefore a cycle of length at most 4 is obtained. A contradiction. 

Therefore, there is a vertex $w\in N^+(u)$ such that $|\overrightarrow{N}(w)\cap N^+(u)|\le 3$.
By a similar reasoning, there is a vertex $w^*\in N^-(u)$  such that $|\overleftarrow{N}(w^*)\cap N^-(u)|\le 3$.

Hence 
$(\overrightarrow{N}(w)\setminus N^+(u))\cap (N^-(u)\cup N_{2}(u))=\emptyset$ and $(\overleftarrow{N}(w^*)\setminus N^-(u))\cap (N^+(u)\cup N_{2}(u))=\emptyset$. Since $|\overrightarrow{N}(w)|=|\overleftarrow{N}(w^*)|=13$, it follows that $|V(G)|\ge 50$, and the result follows.
\end{proof}

\subsection{Upper bounds}

\subsubsection{A family of $[z,r;5]$-mixed graph}

To construct this family of mixed graphs we use the incidence graph of a partial plane. A \textit{partial plane} is defined as two finite sets $\mathscr{P}$ and $\mathscr{L}$ called \emph{points} and \emph{lines}, respectively, where $\mathscr{L}$ consists of subsets of $\mathscr{P}$, such that any line is incident with at least two points, and two points are incident with at most one line. 
The \textit{incidence graph} of a partial plane is a bipartite graph with partite sets $\mathscr{P}$ and $\mathscr{L}$ where a point of $\mathscr{P}$ is adjacent to a line of $\mathscr{L}$ if they are incident. Observe that the incidence graph of a partial plane has even girth $g \geq 6$. In Remark \ref{remark1} we describe a biaffine plane. 

\begin{remark}\label{remark1}\cite{Brown}.

Let $\mathbb{F}_q$ be the finite field of order $q$.
\begin{enumerate}
\item[(i)] Let $\mathscr{L} =\mathbb{F}_q \times \mathbb{F}_q$ and $\mathscr{P} =\mathbb{F}_q \times \mathbb{F}_q$ denoting the elements of $\mathscr{L}$ and $\mathscr{P}$ using ``brackets'' and ``parenthesis'', respectively.
The following set of $q^2$ lines define a biaffine plane:

\begin{equation}
[m, b] = \{(x, mx + b) : x \in \mathbb{F}_q\} \  for \ all\ m, b \in \mathbb{F}_q.
\end{equation}

\item[(ii)] The incidence graph of the biaffine plane is a bipartite graph $B_q = (\mathscr{P},\mathscr{L})$ which is $q$-regular, has order $2q^2$, diameter 4 and girth 6, if $q \geq 3$; and girth 8, if $q=2$.

\item[(iii)] The vertices mutually at distance 4 are the vertices of the sets $L_m = \{[m, b] : b \in \mathbb{F}_q\}$, and $P_x = \{(x, y) : y \in \mathbb{F}_q\}$ for all $x,m \in \mathbb{F}_q$.
\end{enumerate}
\end{remark}

Next, we describe two operations that we perform on the graph $B_q$: reduction and amalgam.

The reduction operation refers to delete the last pairs of blocks $(P_i,L_i)$ from $B_q$. Let $\gamma\in \{1,\ldots, q-1\}$, define $B_q(\gamma)=B_q-\bigcup_{i=1}^{\gamma}(P_{q-i}\cup L_{q-i})$.

\begin{lemma}\cite{Articulo_5}
Let $\gamma\in \{1,\ldots, q-1\}$. Then, the graph $B_q(\gamma)$ is $(q-\gamma)$-regular of order $2(q^2 - q\gamma)$ and girth $g\geq 6$.
\end{lemma}

Now, we describe the amalgam operation. Let $\Gamma_1$ and $\Gamma_2$ be two graphs of the same order and with the same labels on their vertices. The amalgam of $\Gamma_1$ into $\Gamma_2$ is a graph obtained adding all the edges of $\Gamma_1$ to $\Gamma_2$.

We will show how we apply the operations of reduction and amalgam to the graph $B_q$ to build a family of $[z; r; 5]$ -mixed graph.

Let $z\neq 2$ be a positive integer. Let $p$ be the smallest prime number such that $4z+1 \leq p \leq 5z$.
Consider $\mathbb{F}_p=\mathbb{Z}_p$ and let $B_p$ be the incidence graph of the biaffine plane with this field.
Let $\overrightarrow{C}_p (1,\ldots,z)$ the circulant digraph of order $p$. 
Recall that a circulant digraph over $\mathbb{Z}_p$, denoted by $\overrightarrow{C}_p (1,\ldots,z)$ is a digraph whose set of vertices are the elements of $\mathbb{Z}_p$ and the set of arcs in $\overrightarrow{C}_p (1,\ldots,z)$ are defined as $A(\overrightarrow{C}_p (1,\ldots,z))=\{\overrightarrow{ij}|(j-i)\in \{1,\ldots, z\}$ mod $p \}$.
Let $\{1,\ldots, z\}$ be the  \emph{weight} or the  \emph{Cayley  color} of the arcs on $\overrightarrow{C}_p (1,\ldots,z)$.

We define $B^*_p(\gamma)$ to be the amalgam of $\overrightarrow{C}_p (1,\ldots,z)$ into $P_i$ and $L_i$ for $i\in \{0, \ldots, p-\gamma-1\}$ and $\gamma\in \{1, \ldots, p-2\}$. 
To simplify notation, we assume that the labelling of $\overrightarrow{C}_p (1,\ldots,z)$ corresponds to the second coordinate of $P_i$ and $L_i$ for $i\in \{1, \ldots, p-\gamma-1\}$.

To prove Theorem \ref{construction1}, we use a result due to Dusart \cite{Dusart}. For any integer $z\geq 3275$, there always exists a prime number between $n$ and $(1+1/(2ln^2n))n$. Since $(1+1/(2ln^2(4z+1)))(4z+1)<5z$, we ensure that for any integer $z\geq 3275$, there always exists a prime number between $4z+1$ and $5z$. And, it is not difficult to obtain, using simple computer calculations, the same result for $1\leq z \leq 3275$ and $z\neq 2$. 

\begin{theorem}\label{construction1}
Let $p$ be the smallest prime number such that $4z+1 \leq p \leq 5z$, for every positive integer $z\neq 2$. Then $n[z,r;5] \leq 2pr$, for $r \in \{1,\ldots,p\}$.
\end{theorem}

\begin{proof}
Let $p$ be the smallest prime number such that $4z+1 \leq p \leq 5z$, for every positive integer $z$ other than 2.
Let $\overrightarrow{C}_p (1,\ldots,z)$ be the circulant digraph. Observe that $\overrightarrow{C}_p (1,\ldots,z)$ has girth  $5$.
Let $B^*_p(\gamma)$  be the amalgam of $\overrightarrow{C}_p (1,\ldots,z)$ into $P_i$ and $L_i$ for $i\in \{0, \ldots, p-\gamma-1\}$ and $\gamma\in \{1, \ldots, p-2\}$. 
Let $r=p-\gamma$.
Notice that $|B^*_p(\gamma)|=2p(p-\gamma)=2pr$, also each vertex $v \in V(B^*_p(\gamma))$ is $r$-regular in edges and $z$-regular in arcs.
Let $C$ be a shortest cycle in $B^*_p(\gamma)$. Suppose by contradiction, that $|V(C)|\leq 4$. Therefore, $C=(w,x,y,w)$ or $C=(v,w,x,y,v)$. Notice that $C$ cannot be completely contained in $\overrightarrow{C}_p (1,\ldots,z)$ or in $B_p(\gamma)$. With out loose of generality suppose that 
$w,x \in P_i$ and $y\in L_m$ for some $i,m \in \{0,1,\ldots,r \}$, that is, $w=(i,a)$, $x=(i,b)$ and $y=[m,k]$. Since $\overrightarrow{wx} \in A(B^*_p(\gamma)[P_i])$, then $b=a+s$, for some $s \in \{1,\ldots, z\}$.
Since the edges between $P_i$ and $L_m$ induces a matching, then $wy \notin E(B_p(\gamma))$, and hence $wy \notin E(B^*_p(\gamma))$. Thus $|V(C)|>3$, and we can assume $|V(C)|=4$ and $C = (v,w,x,y,v)$.
By the same argument, $v \notin P_i$.
Since there are no edges between $P_i$ and $P_j$ in $B^*_p(\gamma)$, for $j \in \{0,1,\ldots,r \} \setminus  \{i\}$, neither between $L_m$ and $L_n$ in $B^*_p(\gamma)$, for $n \in \{0,1,\ldots,r \}\setminus\{m\}$. Thus $v=[m,l] \in L_m$ and $\overrightarrow{yv} \in A(B^*_p(\gamma)[L_m])$, then $l=k+t$, for some $t \in \{1,\ldots, z\}$. 
If $xy\in E(B_p)$, then $x=(i,b)=(i, a+s)$ and $y=[m,k]$.  Hence,
$a+s=mi+k$, implying  $y=[m,a+s-mi]$. Since $vw \in E(B_p)$, $v=[m,l]$ and $w=(i,a)$ it follows that  $a=mi+l$, that is, $v=[m,a-mi]$. By definition of $\overrightarrow{C}_q (1,\ldots,z)$, we have that  $\overrightarrow{vy} \in A(\overrightarrow{C}_p (1,\ldots,z))$ instead of $\overrightarrow{yv}$, a contradiction.
Hence $g(B^*_p(\gamma))=5$ and $B^*_p(\gamma)$ is a $[z,r;5]$-mixed graph of order $2pr$. Therefore, $n[z,r;5] \leq 2pr$.

If $r=p$, then simply amalgam $\overrightarrow{C}_p (1,\ldots,z)$ in each $P_i$ and $L_i$ of $B_p$, and by a similar analysis  we verify that $g(B^*_p)=5$, it follows that $n[z,r;5] \leq 2pr$.
\end{proof}

In  Figure \ref{gmbq} is depicted an example of a $[3,13;5]$-mixed graph that is an amalgam of the circulant digraph $\overrightarrow{C}_{13} (1,2,3)$ in $B_{13}$.

\begin{figure}[h]
\centering
\includegraphics[scale=.4]{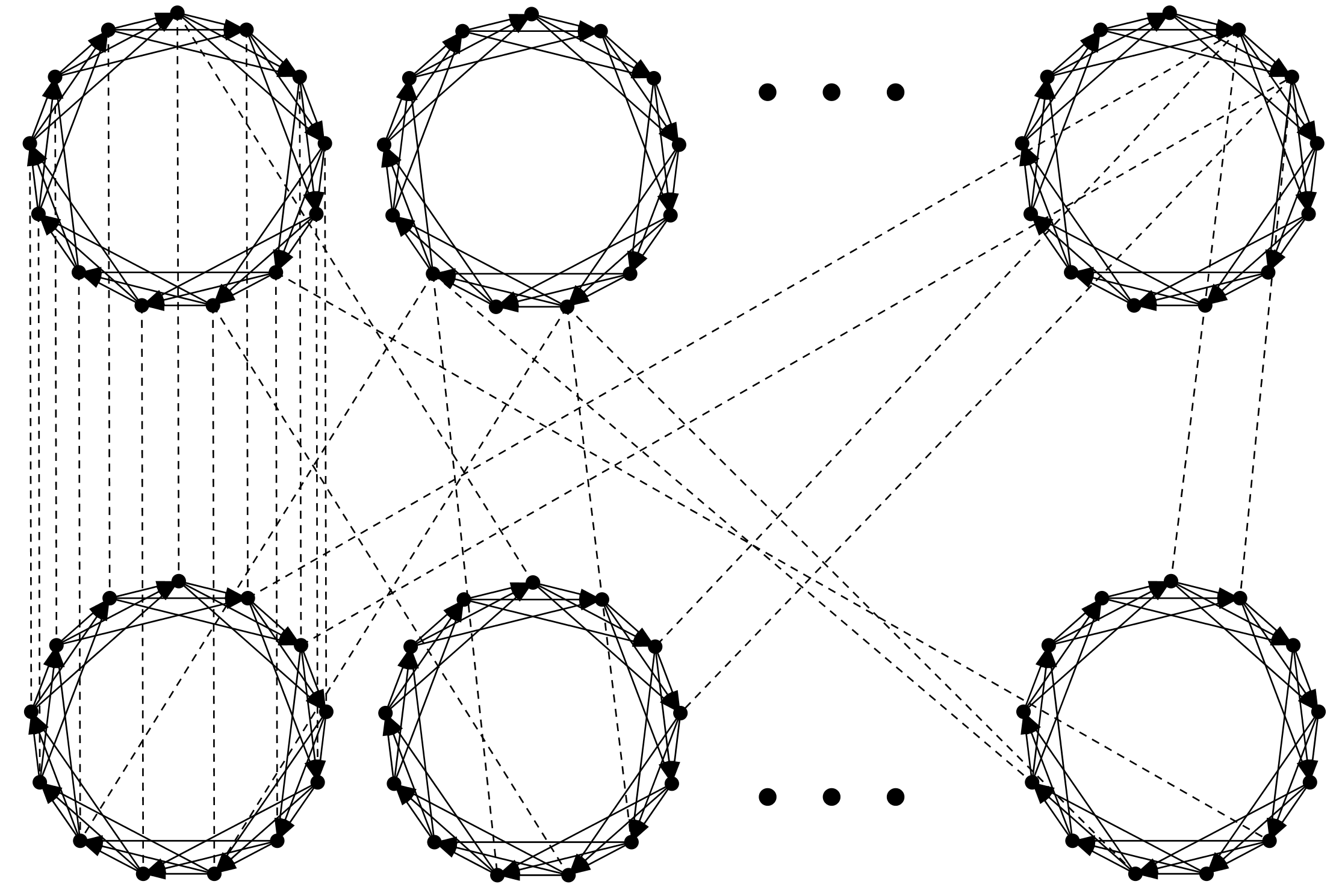}
\caption{A $[3,13;5]$-mixed graph.}
\label{gmbq}
\end{figure}

\subsubsection{Other bounds for different girth}\label{1035}

In this subsection we present an upper bound for $n[z,r;g]$.

\begin{theorem}\label{upper_bound}
Let $r\ge 2$ and $g\ge 3$ be  integers. Then $n[z',r;g]\leq gn_0(r,g)$, for $1\le z'\le n_0(r,g)$.
\end{theorem}

\begin{proof}
We present a construction of a $[z',r;g]$-mixed graph of order $gn_0(r,g)$, for every $z'\in \{1,2,\dots ,n_0\}$.
Let $H$ be an $(r,g)$-cage. Suppose that $V(H)=\{1,2, \ldots ,n_0(r,g)\}$. 
Let $H_i$ be a copy of $H$, for $i \in \{0,1,\dots, g-1\}$, with $V(H_i)=\{1_i,2_i, \ldots ,n_{0,i}\}$. 
Let $G=\bigcup _{i=0}^{g-1}H_i$. Notice that $G$ is a disconnected $r$-regular graph with girth $g$ and order $gn_0(r,g)$.
Let $B(H_i,H_{i+1})$ be the complete bipartite directed graph that is obtained by adding all the arcs from $V(H_i)$ to $V(H_{i+1})$, for $i\in \{0,\dots,  g-1\}$ (mod $g$).
Let ${\cal{A}}_{i}=\{A_1,A_2,\dots , A_{n_0(r,g)}\}$ be a $1$-factor (oriented) of $B(H_i,H_{i+1})$.
Let $\overrightarrow{{\cal{A}}_i}(j)=\bigcup_{i=1}^j {\cal{A}}_i$ and
let $G^*=G+\bigcup_{i=0}^{g-1}\overrightarrow{{\cal{A}}_i}(z')$. Observe that we can always get that ${\cal{A}}_i \cap {\cal{A}}_j=\emptyset$, it follows that  $G^*$ is a  $[z',r;g]$-mixed graph.

In the following we will prove that $G^*$ has girth $g$. Suppose that  $C$ is a cycle such that $|V(C)|<g$, hence it cannot contain edges only. Neither can it consist only of arrows, by
the cycles formed of arcs have length a multiple of $g$. Therefore $C$ consists of edges and arcs, that is, if it contains vertices of the copy $i$, then it contains at least one vertex of the copy $i+1$, and according to the direction of the arcs, the minimum distance of the vertices of the copy $i+1$ to any of the copy $i$ is $g-1$, which is a contradiction. Therefore $ G^*$ has a girth $g$ and is a $[z',r;g]$-mixed graph with $gn_0$ vertices, that is, $n[z',r;g]\leq gn_0(r,g)$.
\end{proof}

\begin{corollary}
There exists a $[10,3;5]$-mixed cage of order 50.
\end{corollary}
\begin{proof}
Let $G$ be a $[10,3;5]$-mixed cage. By Theorem \ref{cota_cuello_5} and Theorem \ref{upper_bound}, it follows that  $|V(G)|=50$. 
The mixed graph depicted in Figure \ref{gm50} is a $[10,3;5]$-mixed cage.
\end{proof}
\begin{figure}[h]
\centering
\includegraphics[scale=.20]{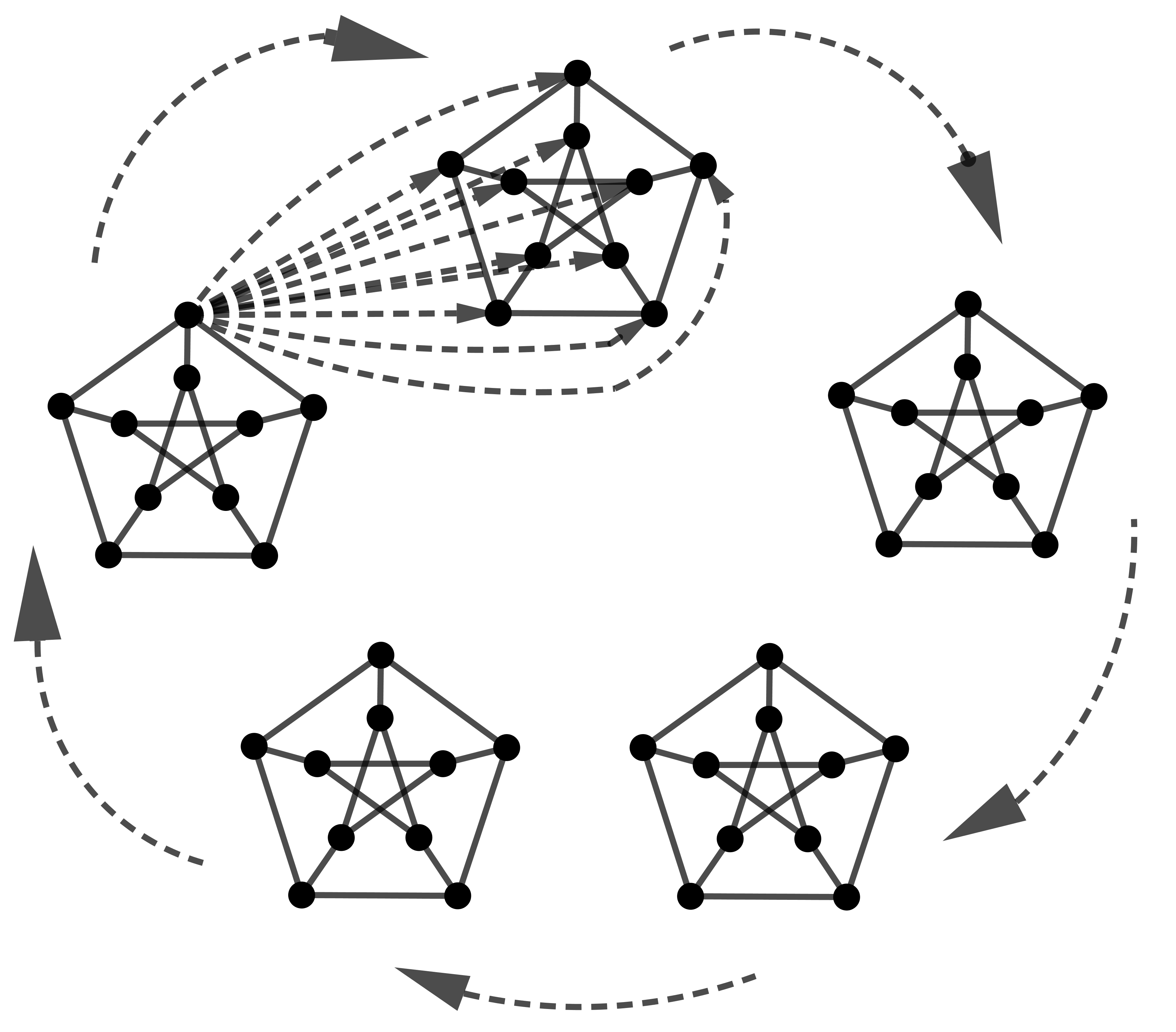}
\caption{A $[10,3;5]$-mixed cage of order 50.}
\label{gm50}
\end{figure}

\section{Future work}

The problem of find a mixed cage and study their properties is very recently.
As a suggestion to continue with the topic we propose two problems:

\begin{enumerate}

\item The study of the monotonocity for $[z,r;g]$-mixed cages with $z\geq 3$.

\item Find better lower upper bounds for $n[z,r;g]$, specially for $g=5$ and also find new constructions of $[z,r;g]$-mixed graphs with few vertices for any $g\geq 5$. A natural suggestion should be study the case for $g=6$.

\end{enumerate}


\end{document}